\documentclass[12pt]{amsart}
\usepackage{amssymb,latexsym,amsmath,amsthm}
\usepackage{graphicx}

\newtheorem{definition}{Definition}[section]
\newtheorem{theorem}[definition]{Theorem}

\newtheorem{lemma}[definition]{Lemma}

\newcommand{\F}{\mathbb{F}}
\newcommand{\fq}{\mathbb{F}_q}
\newcommand{\rmv}[1]{}

\title{value sets of polynomial maps over finite fields}
\author{Gary L. Mullen, Daqing Wan, and Qiang Wang}
\thanks{Research of the authors was partially supported
by NSF and  NSERC of Canada}

\address{Department of Mathematics\\
 The Pennsylvania State University\\
 University Park, PA 16802}
\email{ mullen@math.psu.edu}

\address{Department of Mathematics\\
University of California\\
Irvine,  CA 92697-3875}
\email{dwan@math.uci.edu}

\address{School of Mathematics and Statistics\\
Carleton University\\
1125 Colonel By Drive,\\
Ottawa, ON K1S 5B6\\
Canada}
\email{wang@math.carleton.ca}



\begin{document}

\begin{abstract}
We provide upper bounds for the cardinality of the value set of a polynomial map in several variables 
over a finite field. These bounds generalize earlier bounds for univariate polynomials.
\end{abstract}

\maketitle

\section{Introduction}

Let $\F_q$ be a finite field of $q$ elements with characteristic $p$.
The {\it value set} of a polynomial $f$ over $\fq$ is the set $V_f$ of images  when we view $f$ as a mapping from $\fq$ to itself. Clearly $f$ is a {\it permutation polynomial (PP)} of $\fq$ if and only if the 
cardinality $|V_f|$ of the value set of $f$ is $q$.   As a consequence of the Chebotarev density theorem, 
Cohen \cite{Cohen:70} proved that for fixed  integer $d\geq 1$, there is a finite set $T_d$ of positive rational numbers such that: for any $q$ and any $f\in \fq[x]$ of degree $d$, there is an element $c_f\in T_d$ with $|V_f| = c_f q + O_d(\sqrt{q})$. In particular, when $q$ is sufficiently large compared to $d$, the set of ratios $\frac{|V_f|}{q}$ is contained in a subset of the interval $[0, 1]$ having arbitrarily small measure.
It is therefore natural to ask how the sizes of value sets are explicitly distributed, and also how polynomials are distributed in terms of value sets. 
For example, there are several results on bounds of the cardinality of value sets if $f$ is not a PP over $\fq$;  Wan \cite{W1} proved that $|V_f| \leq q- \lceil (q-1)/d\rceil$ and  Guralnick and Wan \cite{GW} also proved that if $(d, q) =1$ then $|V_f| \leq (47/63) q+ O_d(\sqrt{q})$.  Some progress on lower bounds of $|V_f|$ can be found in \cite{DasMul,WSC}, as well as 
{\it minimal value set polynomials}  that are polynomials satisfying 
 $|V_f| = \lceil q/d \rceil$ \cite{Carlitzetal:61,GomezMadden:88, Mills:64}.
All of  these results relate $|V_f|$ to the degree $d$ of the polynomial. 
Algorithms and complexity in computing $|V_f|$ have been studied recently, see \cite{CHW}.

Let $f: {\F}_q^n\rightarrow {\F}_q^n$ be a 
polynomial map in $n$ variables defined over ${\F}_q$, where $n$ is a positive integer.
In Section~\ref{degree} we extend Wan's result on upper bounds of value sets for univariate polynomials in \cite{W1}
to polynomial maps in $n$ variables.  Denote by $\vert V_f\vert $ the number of 
distinct values taken by $f(x_1, \dots, x_n)$ as $(x_1, \dots, x_n)$ runs over ${\F}_q^n$. Following the approach of studying 
value set problems in terms of the degree of a polynomial, we give  an upper bound of $\vert V_f\vert $  in terms of the total degree of the multivariate polynomial $f$ over $\fq$ in Theorem~\ref{multiValue1}.  In particular, this answers an open problem raised by Lipton \cite{Li} in his computer 
science blog.

\section{Value sets of polynomial maps in several variables}\label{degree}

In this section, we let $f: {\F}_q^n\rightarrow {\F}_q^n$ be a 
polynomial map in $n$ variables defined over ${\F}_q$, where $n$ is a positive integer.
We give a simple upper bound for the number $\vert V_f\vert $ of 
distinct values taken by $f(x_1, \dots, x_n)$ as $(x_1, \dots, x_n)$ runs over ${\F}_q^n$ when
$f$ does  not induce a permutation map.

We  write $f$ as a polynomial vector: 
\begin{equation}\label{defpolyvec}
f(x_1,\dots, x_n)=(f_1(x_1,\dots, x_n), \dots, f_n(x_1,\dots, x_n)), 
\end{equation}
where each $f_i$ ($1\leq i\leq n$) is a polynomial in $n$ variables 
over ${\F}_q$. The polynomial vector $f$ induces a map 
from ${\F}_q^n$ to ${\F}_q^n$. By reducing the polynomial vector $f$ 
modulo the ideal
$(x_1^q-x_1,\dots, x_n^q-x_n)$, we may assume that
the degree of $f_i$ in each variable is at most $q-1$ and we may
further assume that $f$ is a non-constant map to avoid the trivial case.
Let $d_i$ denote  
the total degree of $f_i$ in the $n$ variables $ x_1,\dots, x_n$
and let $d=\max_id_i$. Then $d$ satisfies 
$1\leq d\leq n(q-1)$. Let $\vert V_f\vert $ be the cardinality
of the \emph {value set} $V_f = \{ f(x_1,\dots, x_n) | 
(x_1,\dots, x_n)\in {\F}_q^n \}$. It is clear that
$\vert V_f \vert \leq q^n  $. If $\vert V_f\vert  =q^n $, then $f$ is a \emph {permutation
polynomial vector}, see \cite[Chapter 7]{LN}. If $\vert V_f\vert < q^n $,  we prove the following:

\begin{theorem} \label{multiValue1}
Assume that  $\vert V_f \vert <q^n$. Then
\begin{equation}\label{upperboundD}
 \vert V_f \vert \leq q^n - \min\{ {n(q-1)\over d}, q\}. 
 \end{equation}
\end{theorem}

In the special case when $n=1$, the bound in (\ref{upperboundD}) reduces to the bound (\ref{Wan's bound})
proved in \cite{W1} for the case of a univariate polynomial:

\begin{equation}\label{Wan's bound}
 \vert V_f \vert \leq q -{q-1\over d}. 
\end{equation}

Based on computer
calculations, the bound in (\ref{Wan's bound}) was first conjectured by Mullen \cite{Mu}. 
The original proof of (\ref{Wan's bound}) in \cite{W1} is elementary, and uses power symmetric functions and 
involves a $p$-adic lifting lemma. A significantly simpler proof of (\ref{Wan's bound}) is given 
by Turnwald \cite{Tu}, who uses elementary symmetric functions 
instead of power symmetric functions
and works directly over the finite field ${\F_q}$ without $p$-adic liftings.
Independently and later, Lenstra \cite{Len} showed one of us another simple proof
which uses power symmetric functions in characteristic zero
and avoids the use of the $p$-adic lifting lemma.

The proof of (\ref{Wan's bound}) gives a stronger result as shown in \cite{WSC}. 
This information will be used later to prove the higher dimensional Theorem~\ref{multiValue1}.
We first recall the relevant one dimensional result in \cite{WSC}. 
Let $\mathbb{Z}_q$ denote the ring of $p$-adic integers with uniformizer $p$ and residue field  
${\F_q}$. Let $f$ be a polynomial in $\F_q[x]$ of degree $d>0$.
For a fixed lifting $\tilde{f}(x)\in \mathbb{Z}_q[x] $ of $f$ and a 
fixed lifting
${ L}_q \subset \mathbb{Z}_q$ of ${\F}_q$, we define $U(f)$ to be the smallest positive
integer $k$ such that
\begin{equation} \label{powersum}
S_k(f) =\sum_{x\in {L}_q} \tilde{f}(x)^k \not\equiv 0~({\rm mod}~pk). 
\end{equation}
The number $U(f)$ exists (see the proof of Lemma~\ref{upperboundD2} below) and is 
easily seen to be independent of the choice of the
liftings $\tilde{f}(x)$ and ${L}_q$. One checks from the definition that
$U(f)\geq (q-1)/d$. Thus, we have the inequality, 
$$ \frac{q-1}{d} \leq U(f) \leq q-1.$$
The following improvement of (\ref{Wan's bound})
is given in \cite{WSC}: 

\begin{lemma}\label{upperboundD2}
  If $\vert V_f\vert <q$, then
\[\vert V_f\vert \leq q-U(f). \]
\end{lemma}
 
\begin{proof}
To be self-contained, we give a simpler proof of this lemma using  ideas of
Lenstra and Turnwald,  closely following the version given by Lenstra \cite{Len}.
Note that in this lemma we are dealing with a polynomial $f$ in
one variable.

Let $w=q-\vert V_f\vert $. Assume $\vert V_f\vert  >q-U(f)$, 
that is, $w<U(f)$,  where we define $U(f)=\infty$
if it does not exist. We need to prove that $f$ is bijective on $\F_q$.
By the definition of $U(f)$ and the assumption $w<U(f)$, we can write
$$\sum_{k=1}^{\infty} {S_k(f) \over k} T^k \equiv pg(T) ~({\rm mod}~T^{w+1})$$
for some polynomial $g\in \mathbb{Z}_q[T]$.
This together with the logarithmic derivative identity 
$$\prod_{x\in { L}_q} (1-\tilde{f}(x)T) =\exp( -\sum_{k=1}^{\infty} 
{S_k(f)\over k} T^k ) $$
shows that
$$\prod_{x\in { L}_q} (1-\tilde{f}(x)T)
\equiv \exp( -pg(T))~ ({\rm mod}~T^{w+1}) 
\equiv 1~({\rm mod}~(p, T^{w+1})),$$
where in the last congruence we used the fact that
$p^k/k!$ is divisible by $p$ for every positive integer $k$. Reducing this
congruence modulo $p$,
one obtains
$$\prod_{x\in {\F}_q} (1-f(x)T) \equiv 1 ~({\rm mod}~T^{w+1}).$$
On the other hand, since $f$ is not a constant, we have $w<q-1$ and 
$$\prod_{y\in {\F}_q} (1-yT) = 1-T^{q-1} \equiv 1~({\rm mod}~T^{w+1}).$$
Thus,
$$\prod_{x\in {\F}_q} (1-f(x)T) \equiv \prod_{y\in {\F}_q }(1-yT)
~({\rm mod}~T^{w+1}).$$
By hypothesis, the two products have exactly $\vert V_f\vert $ factors in common. 
Removing the $\vert V_f\vert $ common factors which are invertible modulo $T^{w+1}$,
we obtain two polynomials of degree at most $w$ which are congruent
modulo $T^{w+1}$, and therefore identical. Multiplying
the removed factors back in, we conclude that
$$\prod_{x\in {\F}_q} (1-f(x)T)=\prod_{y\in {\F}_q }(1-yT).$$
This proves that $f$ is bijective on $\F_q$
as required. 
\end{proof}

We use Lemma~\ref{upperboundD2} to prove Theorem~\ref{multiValue1}.  Recall that
$f$ is now the polynomial vector in (\ref{defpolyvec}).
Let $e_1,\dots, e_n$ be a basis of the extension field ${ \F}_{q^n}$
over $ \F_q$. Write $x=x_1e_1+\cdots +x_ne_n$ and 
$$g(x)=f_1(x_1,\dots, x_n)~e_1 +\cdots +f_n(x_1,\dots, x_n)~ e_n. $$
The function $g$ induces a non-constant univariate polynomial map from the finite field
${\F}_{q^n}$ into itself.  
Furthermore, one has the equality 
$\vert V_f\vert =|g({\F}_{q^n})|$. 
We do not have a good control on the degree of $g$ as a univariate polynomial and thus we cannot use 
the univariate bound (\ref{Wan's bound}) directly. The following lemma gives a lower bound 
for $U(g)$, which is enough to prove Theorem~\ref{multiValue1}.

\begin{lemma} If $d\geq n$, we have the inequality 
$$\frac{n(q-1)}{d} \leq U(g) < q^n.$$
If $d< n$, we have the inequality 
$$ q \leq U(g) < q^n.$$
\end{lemma}

{\sl Proof}.  The upper bound is trivial. We need to prove the lower bound. 
We may assume that $g(x_1e_1+\cdots +x_ne_n)$ is already lifted to
characteristic zero and has total degree $d$ when viewed as a polynomial 
in the $n$ variables $ x_1,\dots, x_n$. Furthermore, we can assume that the 
coefficients of $g$ as a polynomial in $n$ variables are either zero or roots 
of unity, that is, we use the Teichm\"uller lifting for the coefficients. Let $L_q$ denote the 
Teichm\"uller lifting of $\mathbb{F}_q$. 

Let $k$ be a positive integer such that $k<n(q-1)/d$ if $d\geq n$ and $k< q$ if $d< n$. We need to prove the claim that 
$$S_k(g)=\sum_{(x_1,\cdots, x_n)\in { L}_q^n} g(x_1e_1+\cdots +x_ne_n)^k \equiv
0~({\rm mod}~pk).$$
Expand 
$g(x_1e_1+\cdots +x_ne_n)^k$
as a polynomial in the $n$ variables $x_1,\dots, x_n$. 
Let 
$$M(x_1,\ldots, x_n) = ax_1^{u_1}\cdots x_n^{u_n}$$
be a typical non-zero monomial in $g^k$. It suffices to prove that 
$$\sum_{(x_1,\ldots, x_n) \in L_q^n} x_1^{u_1}\cdots x_n^{u_n} \equiv  0 \pmod{pk}.$$
The sum on the left side is zero if one of the $u_i$ is not divisible by $q-1$. 
Thus, we shall assume that all $u_i$'s are divisible by $q-1$. 
The total degree 
$$u_1+\cdots +u_n \leq dk. $$ 
Thus, there are at least $n -\lfloor dk/(q-1)\rfloor$ of the $u_i$'s which are zero.  
This implies that 
$$S_k(g) \equiv 0~({\rm mod}~ q^{ n - \lfloor dk/(q-1) \rfloor}).$$ 
Let $v_p$ denote the $p$-adic valuation satisfying $v_p(p)=1$.  If the inequality 
$$v_p(q) ( n - \lfloor kd/(q-1)\rfloor) \geq 1 + v_p(k)$$ 
is satisfied, then the claim is true and we are done.  

In the case that $d< n$ and $k<q$, we have $dk/(q-1) < n$ and $v_p(k) < v_p(q)$. Thus, 
$$v_p(q) ( n - \lfloor kd/(q-1) \rfloor) \geq v_p(q) \geq 1 + v_p(k).$$ 
In the case $d\geq n$ and $k < n(q-1)/d$, we have 
$$k < \frac{n(q-1)}{d} < q.$$
It follows that $v_p(k)< v_p(q)$. 
Since $kd/(q-1) < n$,  we deduce 
$$v_p(q) ( n - \lfloor kd/(q-1) \rfloor) \geq v_p(q) \geq 1 + v_p(k).$$ 
The proof is complete.   $\square$

{\bf Remark}.  For a sharp example, we may take $n=d=2$
and $f(x_1,x_2)=(x_1, x_1x_2)$. This is a birational morphism from $\mathbb{A}^2$ to $\mathbb{A}^2$, 
but not a finite  morphism.  Asymptotic upper bounds for value sets of non-exceptional finite morphisms are given 
in  \cite{GW}.


\begin{thebibliography}{100}
 

 
%\bibitem{AGW} A. Akbary, D. Ghioca, and Q. Wang, On permutation polynomials of prescribed shape,
%{\it Finite Fields Appl.} 15 (2009),  195-206.

\bibitem{Carlitzetal:61}
L. Carlitz,  D. J. Lewis, W. H. Mills, and E. G.  Straus, {Polynomials over finite fields with minimal value sets}, 
 {\it Mathematika} 8 (1961), 121-130.

  
 \bibitem{CHW} Q. Cheng, J. Hill and D. Wan,  Counting value sets: algorithms and complexity, Tenth Algorithmic Number Theory Symposium ANTS-X, 2012.   
   

\bibitem{Cohen:70} S. D. Cohen, The distribution of polynomials over finite fields,  {\it Acta Arith.} 17 (1970), 255-271. 


\bibitem{DasMul} P. Das and G. L. Mullen, Value sets of polynomials over finite fields,  in {\it Finite Fields with Applications in Coding Theory, Cryptography and Related Areas},  G.L. Mullen, H. Stichtenoth, and H. Tapia-Recillas, Eds., Springer, 2002, 80-85.

\bibitem{GomezMadden:88} J. Gomez-Calderon and D. J. Madden, 
Polynomials with small value set over finite fields,  {\it J. Number Theory}  28 (1988), no. 2, 167-188. 

\bibitem{GW}R. Guralnick and D. Wan, Bounds for fixed point 
free elements in a transitive group and applications to curves over finite fields,
{\it Israel J. Math.} 101(1997),  255-287.

\bibitem{Len} H. W. Lenstra, Jr., private communication to Daqing Wan.

\bibitem{LN}R. Lidl and H. Niederreiter, {Finite Fields}, Sec. Ed.,  Cambridge University Press, Cambridge, 1997.


\bibitem{Li} R. Lipton,  Claiming Picard's math may have gaps, 
http://rjlipton.wordpress.com/2011/09/26/claiming-picards-math-may-have-gaps/. 

%\bibitem{MZ}
%A. Masuda and M. E. Zieve, Permutation binomials over finite fields,  {\it Trans. Amer. Math. Soc}.  361 (2009), No. 8, 4169-4180. 

\bibitem{Mills:64} W. H. Mills,  Polynomials with minimal value sets, 
{\it Pacific J. Math} 14 (1964), 225-241. 

\bibitem{Mu}G. L. Mullen, Permutation polynomials over finite fields, 
Lecture Notes in Pure and Appl. Math., Vol. 141, Marcel Dekker,
New York, 1992,  131-151.


\bibitem{Tu}G. Turnwald, A new criterion for permutation polynomials,  {\it Finite Fields
 Appl.} 1(1995), 64-82.

\bibitem{W1}D. Wan, A $p$-adic lifting lemma and its applications to permutation
polynomials, Lecture Notes in Pure and Appl. Math., Vol. 141, Marcel Dekker,
New York, 1992, 209-216.






\bibitem{WSC}D. Wan, P. J. S. Shiue and C. S. Chen, Value sets of polynomials over
finite fields, {\it Proc. Amer. Math. Soc}. 119(1993), 711-717.   
  



%\bibitem{Wang}
%Q. Wang, \emph{Cyclotomic mapping permutation polynomials over finite fields}, Sequences, Subsequences, and 
%Consequences (International Workshop, SSC 2007, Los Angeles, CA, USA, May 31 - June 2, 2007), 
% Lecture Notes in Comput. Sci. 4893, 119--128. 



\end{thebibliography}
\end{document}